\documentclass[conference]{IEEEtran}

\usepackage[dvips]{graphicx}
\usepackage{amsmath,amssymb}
\usepackage{amsthm}
\usepackage{algorithm}
\usepackage{algorithmic}
\usepackage[english]{babel}

\newtheorem{theorem}{Theorem}
\newtheorem{lemma}[theorem]{Lemma}
\newtheorem{corollary}[theorem]{Corollary}

\newtheorem{conjecture}[theorem]{Conjecture}
\newtheorem{definition}[theorem]{Definition}
\newtheorem{example}[theorem]{Example}

\begin{document}
\title{On spherical 4-distance 7-designs}
\date{}

\author{
\IEEEauthorblockN{Peter Boyvalenkov}
\IEEEauthorblockA{Institute of Mathematics and Informatics, \\ Bulgarian Academy of Sciences, \\
8 G Bonchev Str.,
1113  Sofia, Bulgaria \\ peter@math.bas.bg}\medskip
\and
\IEEEauthorblockN{Navid Safaei}
\IEEEauthorblockA{Research Institute of  Policy Making, \\ Sharif University of Technology, \\
Tehran, Iran \\
navid\_safaei@gsme.sharif.edu}}
\maketitle

\begin{abstract}
We investigate spherical 4-distance 7-designs by studying their distance distributions.
We compute these distance distributions and use their product (an integer) 
to derive certain divisibility conditions relating the dimension $n$ and the cardinality $M$ of
our designs. It follows that $n$ divides $12M$ and $n+1$ divides $4M^2$. This result  
provides a good base for computer experiments to support the 
folklore conjecture that the only spherical 4-distance 7-designs are the tight spherical 7-designs.
We then proceed with a computer assisted proof of this conjecture in all dimensions $n \leq 1000$.


\end{abstract}

\section{Introduction}

Let $\mathbb{S}^{n-1}$ be the unit sphere in $\mathbb{R}^n$. A finite set $C \subset \mathbb{S}^{n-1}$ is called a spherical code. 
Spherical designs were introduced by Delsarte, Goethals and Seidel \cite{DGS} as a special class of spherical codes with good combinatorial and integration properties.

\begin{definition} \label{def-designs} \cite{DGS} A spherical code $C \subset \mathbb{S}^{n-1}$ is called a spherical $\tau$-design if the quadrature 
formula
\[  \int_{\mathbb{S}^{n-1}} f(x) d\mu(x) =
                  \frac{1}{|C|} \sum_{x \in C} f(x) \]
is exact for all polynomials $f(x)= f(x_1,x_2,\ldots,x_n)$ of degree at most $\tau$. Here the measure
$\mu$ is normalized, i.e. $\mu(\mathbb{S}^{n-1})=1$. The maximal positive integer $\tau$ such that $C$ is a spherical $\tau$-design
is called strength of $C$. 
\end{definition}

A spherical $\tau$-design  $C \subset \mathbb{S}^{n-1}$  is called {\it tight} if it attain the Delsarte-Goethals-Seidel bound \cite{DGS}
\[ |C| \geq {n+k-1-\varepsilon \choose n-1}+{n+k-2 \choose n-1}, \]
where $\tau=2k-\varepsilon$, $\varepsilon \in \{0,1\}$. 

For a spherical code $C$ we consider the set of distinct inner products of points of $C$,
\[ A(C):=\{ \langle x,y \rangle : x,y \in C, x \neq y\}, \]
and denote by $s:=|A(C)|$ their number. The code $C$ is called then an $s$-distance set.

Designs with large strength $\tau$ and small number of distances $s$ are clearly interesting. It is well known \cite{DGS} that $\tau \leq 2s$
or even $\tau \leq 2s-1$ if $A(C) \cup \{1\}$ is symmetric with respect to 0. Levenshtein \cite{Lev92} proved that if 
$\tau \geq 2s-1$ or $\tau \geq 2s-2$ and $-1 \in A(C)$, then $C$ is maximal (and attains what is called now Levenshtein bounds).

In this paper we consider 
the case $(s,\tau)=(4,7)$. We prove that all 4-distance 7-designs have cardinality which has some good divisibility 
properties. We prove that the dimension $n$ divides $12M$ and, in some cases even better, $n$ divides $M$. 

All these results are collected towards the following conjecture which could be quite old but we have not seen it explicitly written. 

\begin{conjecture} \label{conj-bannai}
Let $C \subset \mathbb{S}^{n-1}$, $n \geq 2$, be a spherical 4-distance 7-design. Then $C$ is a tight spherical 7-design. In particular,
$|C|=2{n+2 \choose 3}$. 
\end{conjecture}

Tight spherical 7-designs are possible only in dimensions $n=3k^2-4$, $k \geq 2$, and are known only for $k=2$ and $k=3$, where the 
corresponding designs are unique up to isometry of $\mathbb{S}^{n-1}$. Further nonexistence results for tight spherical 7-designs were 
obtained in \cite{BMV04,NV13}. 

In this paper we provide a computer assisted proof of Conjecture \ref{conj-bannai} in all dimensions $n \leq 1000$ based on the 
results from Section 3. This result is confirmed in dimensions $n \leq 215$ by a brute force approach based on formulas 
from Section 2.

\section{Preliminaries}

We collect some facts about 4-distance 7-designs which follow from general results in \cite{DGS,Lev92,BDL}. 
The exposition is similar to \cite{BS20}, where we considered 3-distance 5-designs. 

Let $C \subset \mathbb{S}^{n-1}$ be a spherical 4-distance 7-design with  cardinality $M=|C|$ and
\[ A(C)=\{a,b,c,d\}, \]
where $-1 \leq a<b<c<d<1$. For $x \in C$ let 
\[ (X,Y,Z,T)=(A_a(x),A_b(x),A_c(x),A_d(x)) \] be the distance distribution of $C$ with respect of $x$, i.e.,
\[ A_a(x)=\left|\{y \in C: \langle x,y \rangle =a \}\right|, \]
etc. It is well known \cite{DGS} that the distance distribution does not depend on $x$ (this fact follows whenever
$s-1$ does not exceed $\tau$).

Then the nine numbers $a$, $b$, $c$, $d$, $X$, $Y$, $Z$, $T$, and $|C|=M$ satisfy the following system of eight equations
\begin{equation} \label{syst1}
a^iX+b^iY+c^iZ+d^iT=f_i M-1, \ \ i=0,1,\ldots,7, 
\end{equation}
where $f_i=0$ for odd $i$, $f_0=1$, $f_2=1/n$, $f_4=3/n(n+2)$ and $f_6=15/n(n+2)(n+4)$. The cardinality $M$ is restricted by
\begin{equation} \label{M-range}
2 {n+2 \choose 3} \leq M \leq {n+3 \choose 4}+{n+2 \choose 3}.
\end{equation}
The lower bound in \eqref{M-range} comes from $C$ being a 7-design, and the upper bound (so-called absolute bound) 
follows from $C$ being a 4-distance 
set. Both bounds were proved in \cite{DGS}. We will assume that $C$ is not a tight spherical 7-design, i.e. 
\[ M>2{n+2 \choose 3}=\frac{n(n+1)(n+2)}{3}. \] 

In what follows we will use for short the notations 
\[ A= 6M-n(n+1)(n+5), \]
\[ B=3M-n(n+1)(n+2). \] 

The numbers $a,b,c,d$ are the roots of the equation 
\begin{equation} \label{lev_equ}
f(t)=0,
\end{equation}
where
\begin{align*}
f(t)=& (n+2)(n+4)At^4+n(n-1)(n+1)(n+2)(n+4)t^3 \\
& -9(n+2)(4M-n(n+1)(n+3))t^2 \\
&-3n(n-1)(n+1)(n+2)t+6B. 
\end{align*}
This fact already involves the Levenshtein framework from \cite{Lev92,Lev98} and the polynomial $f(t)$ is, in the notations from
\cite{Lev92}, equal to
\[ (t-\alpha_0)(t-\alpha_1)(t-\alpha_2)(t-\alpha_3) \]
up to a multiplicative constant, where $\alpha_0=a$, $\alpha_1=b$, $\alpha_2=c$, $\alpha_3=d$, in a turn to our notations. 
We also note that $b<0<c$ and 
\begin{equation} \label{ineq-inner}
|a|>|d|>|b|>|c|
\end{equation}
(see \cite{BDL} for the general case of $s$-distance $(2s-1)$-designs). 

It is straightforward to resolve the Vandermonde-like system of equations from \eqref{syst1} with odd $i=1,3,5,7$. 
The solution
\begin{eqnarray*}
X &=& -\frac{(1-b^2)(1-c^2)(1-d^2)}{a(a^2-b^2)(a^2-c^2)(a^2-d^2)}, \\
Y &=& -\frac{(1-a^2)(1-c^2)(1-d^2)}{b(b^2-a^2)(b^2-c^2)(b^2-d^2)}, \\
Z &=& -\frac{(1-a^2)(1-b^2)(1-d^2)}{c(c^2-a^2)(c^2-b^2)(c^2-d^2)}, \\
T &=& -\frac{(1-a^2)(1-b^2)(1-c^2)}{d(d^2-a^2)(d^2-b^2)(d^2-c^2)} \\
\end{eqnarray*}
 is unique because of the inequalities \eqref{ineq-inner} (see \cite{BDL} for a treatment of the general case of $s$-distance $(2s-1)$-designs). 

\section{Divisibility results for $M$}

\subsection{Computation of $XYZT$ in terms of $n$ and $M$}

The integers $X$, $Y$, $Z$, and $T$ can be involved in different symmetric expressions. We tried many reasonable
combinations but succeeded to get results only from the investigation of $XYZT$ as shown below. 

Thus, we consider the product
\[ XYZT = \frac{f(1)^3f(-1)^3}{abcd D F^2}, \]
where 
\[ D=a_0^6\left((a-b)(a-c)(a-d)(b-c)(b-d)(c-d)\right)^2 \]
is the discriminant of $f$, 
\[ a_0= (n+2)(n+4)A \] 
is its leading coefficient and
\[ F=(a+b)(a+c)(a+d)(b+c)(b+d)(c+d). \]
Expressing the necessary symmetric functions of $a$, $b$, $c$, $d$ obtained via \eqref{lev_equ} we find that
\[ abcd=\frac{6B}{(n+2)(n+4)A}, \]
\[ f(1)f(-1)=a_0^2 \cdot \frac{12(n-1)^2(n+1)^2B}{(n+2)^2(n+4)^2A^2}, \]
\[ \frac{f(1)f(-1)}{F}=a_0^2 \cdot \frac{MA}{n^2(n+1)(n+2)} \]
(the last expression was computed by subtracting the equations with $i=2$ and $i=4$ of the system \eqref{syst1}),  
and the most complicated
\[ D= a_0^6 \cdot \frac{108(n+1)^2R(n,M)}{(n+2)^3(n+4)^5 A^6}, \]
where 
\begin{align*}
& R(n,M) = 2^{14}3^{6}M^6-2^{13}3^7n(n+1)(n+3)M^5 \\
& +2^93^5n^2(n+1)(n^4+93n^3+629n^2+1339n+818)M^4 \\
& -2^73^4n^3(n+1)^2(13n^5+436n^4+3688n^3+12782n^2 \\
& \qquad \qquad +19163n+9998)M^3 \\
& +2^23^3n^4(n+1)^3(n+2)(n+7)^2(5n^4+447n^3+3303n^2\\
& \qquad \qquad +7873n+5652)M^2 \\
& -2^23^3n^5(n+1)^4(n+2)^2(n+5)^2(n+7)^3(3n+5)M \\
& +n^6(n+1)^5(n+2)^3(n+5)^3(n+7)^4.
\end{align*}

Combining these we obtain 
\begin{equation} \label{xyzt}
XYZT=\frac{M^3(n-1)^2(n+4)^4A^7}{54n^4(n+1)^2R(n,M)}.
\end{equation}
We investigate below the integrality of $XYZT$. 

\subsection{Divisibility properties of $M$}

For a prime $p$ and a nonzero integer $x$, we will denote by $v_p(x)$ the largest power of $p$ which divides $x$ (i.e., the power of $p$ in the 
canonical representation of $x$; for example $v_2(24)=3$, $v_3(24)=1$ and $v_5(24)=0$). 

The power of $p$ in the numerator and denominator of XYZT in \eqref{xyzt} will be denoted for short by
$v_p(\mbox{num})$ and $v_p(\mbox{den})$, respectively. Since $XYZT$ is a positive integer, the inequality $v_p(\mbox{den}) \leq v_p(\mbox{num})$ holds
true for every $p$. 

\begin{lemma} \label{lem-1} 
If $XYZT$ is an integer, then the following statements hold true:

\begin{itemize}

\item[a)] if $gcd \left(n,6\right)=1$, then $n$ divides $M$; 

\item[b1)] if $gcd\left(n,3\right)=1,$ $n$ is even but $v_2(n) \not\in \{2,1+v_2(M)\}$, then $n$ divides $M;$  

\item[b2)] if $gcd\left(n,3\right)=1,$ and $v_2\left(n\right)=2$, then $n$ divides $4M;$

\item[b3)] if $gcd\left(n,3\right)=1,$ and $v_2\left(n\right)=1+v_2\left(M\right)$, then $n$ divides $2M;$

\item[c1)] if $gcd(n,2)=1$, $3$ divides $n$ but $v_3\left(n\right)\neq 1+v_3\left(M\right)$, then $n$ divides $M;$

\item[c2)] if $gcd(n,2)=1$ and $v_3\left(n\right)=1+v_3\left(M\right)$, then $n$ divides $3M$; 

\item[d1)] if $6$ divides $n$, $v_3(n) \neq 1+v_3(M),$ and $v_2(n) \not\in \{2,1+v_2(M)\}$, then $n$ divides $M;$

\item[d2)] if $6$ divides $n$, $v_3(n)=1+v_3(M)$, and $v_2(n) \not\in \{2,1+v_2(M)\}$, then $n$ divides $3M;$

\item[d3)] if $6$ divides $n$, $v_3(n)=1+v_3(M)$, and $v_2(n)=1+v_2(M),$ then $n$ divides $6M;$

\item[d4)] if $6$ divides $n$, $v_3(n)=1+v_3(M)$, and $v_2(n)=2$, then $n$ divides $12M.$ 

\end{itemize}

\end{lemma}

\begin{proof} Considering the numerator of $XYZT$ from \eqref{xyzt} modulo $n$ we find that $n$ divides $2^{15}3^7M^{10}$. 
Hence $v_2(n) \leq 10v_2(M)+15$, $v_3(n) \leq 10v_3(M)+7$ and $v_p(n) \leq 10v_p(M)$ for every prime $p>3$. 

Assume that $v_p(n)>v_p(M)$ for some prime $p>3$. Then it follows that $p$ divides $M$, $v_p(A)=v_p(M)$, and we have $v_p(\mbox{num})=10v_p(M)$. 
For the power of $p$ in the denominator we see that $v_p(R(n,M))=6v_p(M)$, whence $v_p(\mbox{den})=4v_p(n)+6v_p(M)$. 
Thus
\[ 10v_p\left(M\right) \geq 4v_p\left(n\right)+6v_p\left(M\right) > 10v_p\left(M\right), \] 
a contradiction. Therefore
$v_p\left(M\right)\ge v_p\left(n\right)$ for all primes $p>3$. This proves a) and makes obvious the statements in b2), b3), c2), d3) and d4). Note also that 
the conclusions modulo 2 in d1) will imply d2). Thus, it remains to prove b1), c1), and d1), where the primes $p=2$ and $p=3$ are only relevant. 

Assume now that $n$ is even, $v_2\left(n\right)\not\in \{1+v_2\left(M\right),2\}$, and $v_3\left(n\right)\neq 1+v_3\left(M\right)$.
We will argue by contradiction, assuming that $v_2\left(n\right) >1+v_2\left(M\right)$. Then it follows that $v_2\left(n+4\right)=2$ and, therefore, 
\[v_2(\mathrm{num})=10v_2\left(M\right)+15.\] 
For the denominator we first see that 
\begin{align*}
v_2&\left(R(n,M)\right)\ge \min \{6v_2\left(M\right)+14,5v_2\left(M\right)+\ v_2\left(n\right)+13,\\
& 4v_2\left(M\right)+2v_2\left(n\right)+10,\ 3v_2\left(M\right)+3v_2\left(n\right)+8,\\
& 2v_2\left(M\right)+4v_2\left(n\right)+5, 5v_2\left(n\right)+v_2\left(M\right)+4, 6v_2\left(n\right)+9\}.
\end{align*}
It is easy to see that $v_2\left(n\right)>1+v_2\left(M\right)$ implies that each term in the minimum is greater than or equal to $6v_2\left(M\right)+13$. 
We conclude that
\[v_2\left(\mathrm{den}\right)\ge 4v_2\left(n\right)+1+6v_2\left(M\right)+13> 10v_2\left(M\right)+18.\] 
Therefore,
\[ 10v_2\left(M\right)+18 < v_2(\text{num}) = 10v_2\left(M\right)+15, \]
a contradiction. Hence $v_2\left(n\right)\le v_2\left(M\right)$, 
which completes, in particular, the proof of b1). 

Finally, assume that $v_3\left(n\right)>1+v_3\left(M\right)$ for the proofs of c1) and d1). Then 
\[v_3(\mathrm{num})=10v_3\left(M\right)+7.\] 
On the other hand,
\begin{align*}
v_3&\left(R(n,M)\right)  \ge \min \{6v_3\left(M\right)+6,\ 5v_3\left(M\right)+v_3\left(n\right)+7,\\ 
& 4v_3\left(M\right)+2v_3\left(n\right)+5,\ 3v_3\left(M\right)+3v_3\left(n\right)+4,\\
& 2v_3\left(M\right)+4v_3\left(n\right)+3,\ v_3\left(M\right)+5v_3\left(n\right)+3,\ 6v_3\left(n\right)\}.
\end{align*}
It is easy to deduce that all terms in the last minimum are greater than or equal to $6v_3\left(M\right)+6$. Therefore,
\[v_3\left(\mathrm{den}\right) \geq 3+4v_3\left(n\right)+6v_3\left(M\right)+7 >10v_3(M)+14,\] 
whence 
\[  10v_3(M)+14 < v_3(\text{num}) = 10v_3\left(M\right)+7, \]
which is impossible. This completes the proof.
\end{proof}

We can bring some parts of the above lemma together. For example we can combine parts  c1, c2 as follows: 
if $\text{gcd}(n, 2)= 1$ and $3|n$, then $n|3M$. More general, we present a reformulation that is focused on the divisibility 
properties of $M$. 

\begin{corollary}  Let $C \subset \mathbb{S}^{n-1}$ be a spherical 4-distance 7-design of cardinality $|C|=M$.
Then $n$ divides $M$ if one of the following is fulfilled: 
\begin{itemize}
\item[a)] $gcd(n,6)=1$;
\item[b)] $2|n$, $gcd\left(n,3\right)=1$, and $v_2(n) \not\in \{2,1+v_2(M)\}$;
\item[c)] $gcd(n,2)=1$, $3|n,$ and $v_3\left(n\right)\neq 1+v_3\left(M\right)$;
\item[d)] $6|n$, $v_3(n) \neq 1+v_3(M),$ and $v_2(n) \not\in \{2,1+v_2(M)\}$.
\end{itemize}
Further, $n$ divides $2M$ if $gcd\left(n,3\right)=1$ and $v_2\left(n\right)=1+v_2\left(M\right)$;
$n$ divides $4M$ if $gcd\left(n,3\right)=1$ and $v_2\left(n\right)=2$; $n$ divides $3M$ if  
$gcd\left(n,2\right)=1$ and $v_3\left(n\right)=1+v_3\left(M\right)$ or $v_3(n)=1+v_3(M)$ and $v_2(n)=1+v_2(M)$;
$n$ divides $6M$ if $v_3(n)=1+v_3(M)$ and $v_2(n)=1+v_2(M)$; and $n$ divides $12M$ if  
$v_3(n)=1+v_3(M)$ and $v_2(n)=2$. In particular, $n$ divides $12M$.
\end{corollary}

We proceed with divisibility with respect to the prime divisors of $n+1$ as instructed from \eqref{xyzt}.  

\begin{lemma} \label{lem-2} 
If $XYZT$ is an integer, then the following statements hold true:

\begin{itemize}

\item[a)] if $gcd (n+1,6)=1$, then $n+1$ divides $M^2$;

\item[b)] if $n+1$ is even and $gcd(n+1,3)=1$, then $n+1$ divides $16M^2$;

\item[c)] if $n+1$ is odd and $3$ divides $n+1$, then $n+1$ divides $3M^2;$

\item[d)] if $6$ divides $n+1$, then $(n+1)^2$ divides $48M^4.$

\end{itemize}

\end{lemma}

\begin{proof} Considering the numerator in \eqref{xyzt} modulo $n+1$ we conclude that $n+1$ divides $2^9.3^{11}M^{10}$.

If a prime $p>3$ divides $n+1,$ then it divides $M$. We assume for a contradiction that
that there exists a prime $p>3$ such that $ v_p(n+1)> 2v_p(M)$. Then the power of $p$ in the numerator is
\[ v_p(\mbox{num})=10v_p\left(M\right). \] 
Further, it also follows from our assumption that $v_p\left(R(n,M)\right) = 6v_p\left(M\right)$, whence  
\[v_p(\mbox{den}) =6v_p\left(M\right)+2v_p\left(n+1\right).\] 
Therefore  
\[ 6v_p\left(M\right)+2v_p\left(n+1\right) \le 10v_p\left(M\right), \] 
which is impossible when $ v_p(n+1)> 2v_p(M)$. 
This completes the proof of a) and leaves the proofs of b), c) and d) only up to 
considerations of the powers of 2 and 3. 

Modulo 2 considerations are needed for b) and d) only. We assume $v_2(n+1)\ge 2v_2(M)+5$ 
for a contradiction. This implies that $v_2(A)=1+v_2(M)$, $v_2(n-1)=1$, and $n+4$ is odd. Therefore 
\[v_2(\mbox{num})=10v_2\left(M\right)+9.\] 
To estimate the power of 2 in the denominator we observe that 
\begin{align*} 
v_2\left(R(n,M)\right) & \ge \min \{6v_2\left(M\right)+14, v_2\left(n+1\right)+4v_2\left(M\right)+9\} \\
& = 6v_2\left(M\right)+14
\end{align*}
to conclude that
\[ v_2(\mbox{den})\ge 6v_2\left(M\right)+2v_2\left(n+1\right)+15.  \]
Hence, 
\[ 2v_2(n+1)+6v_2(M)+15 \leq 10v_2\left(M\right)+9, \]
i.e. $v_2(n+1)+3 \le 2v_2(M)$, which is impossible when 
$v_2(n+1)\ge 2v_2(M)+5$. This completes the proof of b) and what concerns the power of 2 in d).

Finally, in our course to finish the proof in c) and d), we assume that $v_3\left(n+1\right)>1+2v_3\left(M\right)$
for a contradiction. Similarly to above we see that $v_3(A)=1+v_3(M)$, $v_2(n+4)=1$ and $v_3(n-1)=0$, whence 
\begin{align*}
v_3(\mbox{num})& = 3v_3\left(M\right)+7\left(1+v_3\left(M\right)\right)+4 \\
& =10v_3\left(M\right)+11.
\end{align*} 
For the denominator we first see that 
\[ v_3\left(R(n,M)\right)\ge 6v_3\left(M\right)+6 \]
giving that 
\[ v_3\left(\mbox{den}\right) \geq 6v_3\left(M\right)+2v_3\left(n+1\right)+9, \]
which contradicts to our assumption. 
\end{proof}

In our computer investigation below we will use the following reformulation. 

\begin{corollary}
Let $C \subset \mathbb{S}^{n-1}$ be a spherical 4-distance 7-design of cardinality $|C|=M$.
Then $n+1$ divides $4M^2$. 
\end{corollary}

\subsection{Further divisibility conditions}

In \cite{Noz}, Nozaki proved necessary conditions for existence of spherical $k$-distance sets as a generalization 
of the classical Larman-Roges-Seidel result on 2-distance sets \cite{LRS}. We present his result here in our context 
to find another divisibility condition.

The Nozaki's Theorem 5.1 \cite{Noz} states that the number
\[ k_a:=\frac{(1-b)(1-c)(1-d)}{(a-b)(a-c)(a-d)} \]
and its analogs for $b$, $c$ and $d$ are integers which sum up to 1 and their absolute values do not exceed 
\[ \left\lfloor{\frac{1}{2}+\sqrt{\frac{N^2}{2N-2}+\frac{1}{4}}}\right\rfloor,\]
where 
\[ N={n+2 \choose 3}+{n+1 \choose 2} =\frac{n(n+1)(n+5)}{6}. \]

Applying our technique (i.e. employing the equation \eqref{lev_equ}) to this, we obtain 
 \begin{equation} \label{noz-cond}
k_ak_bk_ck_d=\frac{2M^3(n+1)(n+4)^2(n-1)^3A^3}{R(n,M)},
\end{equation}
which should be integer. We were unable to get further divisibility conclusions from \eqref{noz-cond}
but combinations of \eqref{xyzt} and \eqref{noz-cond} could give something. However, the proofs of 
Lemmas 3 and 5 could be probably simplified by using  \eqref{noz-cond} instead of \eqref{xyzt}.

\section{Computer assisted investigations of the integrality conditions}

We describe the computer investigations of the integrality of $XYZT$ based on the divisibility conditions from Lemmas 3 and 5
and general investigations of the integrality of $XYZT$ and $k_ak_bk_ck_d$ from \eqref{noz-cond}. 

For fixed $n \geq 3$ all positive integers $M$ in the range defined by \eqref{M-range} are checked (with a strict inequality 
$M>2{n+2 \choose 3}$ to skip the considerations of tight spherical 7-designs). We first determine whether 
the pair $(n,M)$ satisfies the conditions of Lemmas 3 and 5 (i.e., whether $n|12M$ and $n+1|4M^2$). Then, for each such pair 
we check if $XYZT$ is integer. 
If the pair $(n,M)$ passes this test (i.e., if the number $XYZT$ is integer) we compute the values of $ a,b,c,d$ 
as roots of the equation \eqref{lev_equ} and find the values of $X$, $Y$, $Z$, and $T$ themselves. The last part can be realized via
computation of $k_a$, $k_b$, $k_c$, and $k_d$ as well. 

This computation yields the following (computer-assisted) result. 

\begin{theorem} \label{thm-main}
There exist no spherical 4-distance 7-designs in dimensions $3 \leq n \leq 1000$ if
\[M>2{n+2 \choose 3}.\]
\end {theorem}

In other words, the only possible spherical 4-distance 7-designs  in dimensions $3 \leq n \leq 1000$ are
the tight spherical 7-designs (see the comment after Conjecture 2). 

In these dimensions, there are 321 pairs $(n,M)$ that pass the tests of Lemmas 3 and 5 and, furthermore,  $XYZT$ is an integer. 
This was verified by a simple Python program. There is no dimension with more than two suitable values of $M$ and there are 
12 dimensions with two possible values of $M$ each. 

Further, in all 321 cases none of $X$, $Y$, $Z$, and $T$ is an integer. This was independently verified by a simple Maple program. 
With the last check, the non-integrality of the parameters $k_a$, $k_b$, $k_c$, and $k_d$ was also verified. 

It is up to computing power to proceed with this proof. We used only a general purpose laptop and stopped at $n=1000$ as it 
already started to take about three hours of computation per dimension. It seems that dimensions $n \leq 2000$ are reachable 
by this approach and more powerful computers. All our programs and the most important intermediate data are available upon request.

\begin{example} \label{ex1}
The data from the first case that passes the integrality test of $XYZT$ is as follows. For $(n,M)=(7,196)$ we have that $n=7$ divides $2M=392$ and
$n+1=8$ divides $4M^2=153664$, i.e. the conditions from Lemmas 3 and 5 are satisfied. Then
\[ XYZT=1185921 \] 
from \eqref{xyzt} and 
\[ k_ak_bk_ck_d=121 \] from \eqref{noz-cond}, i.e. our condition and the Nozaki's condition are satisfied and there is no contradiction so far. 
Thus we find the polynomial \eqref{lev_equ} as 
\[ f(t)= 49896t^4+33264t^3-18144t^2-9072t+504, \]
whence the inner products are computed, 
\[ a \approx -0.821721, \ \ b \approx -0.442124,  \]
\[ c \approx 0.0508952, \ \ d \approx 0.546284. \]
This gives (with enough precision) the numbers
\[ X \approx 5.63, \ \ Y \approx 41.57, \ \ Z \approx 93.84, \ \ T \approx 53.95,  \]
clearly non-integer. Therefore the existence question for spherical 4-distance 7-designs on $\mathbb{S}^6$ with 196 points is
resolved in the negative.
\end{example}

We also completed a brute force investigation of both $XYZT$ and $k_ak_bk_ck_d$ in dimensions $n \leq 215$ with a Python program. 
There are 72 pairs that pass the $XYZT$ test and 27 pairs that pass the $k_ak_bk_ck_d$ test, the smallest one being $(n,M)=(7,196)$ 
(it passes both tests, as mentioned in Example \ref{ex1}). It is clear that these 72 pairs are the same as these passing the test via Lemmas 3 and 5 since the lemmas are based
on the investigation of the integrality of $XYZT$. In all these cases we compute $X$, $Y$, $Z$ and $T$ to see that they are not integers, confirming this way Theorem \ref{thm-main}. 
The integrality of  $k_ak_bk_ck_d$ is usually stronger, despite there are two cases that pass the $k_ak_bk_ck_d$ test while failing 
to pass the $XYZT$ test. The disadvantage of the integrality criterion of \eqref{noz-cond} is that one can not start with prime 
divisors of $n$ and $n+1$. Of course, the strongest test is the integrality of $X$, $Y$, $Z$, and $T$ but we do not see other approaches
than brute force investigation of the formulas from Section 2.

\section*{Acknowledgments}
The research of the first author was supported, in part, by a Bulgarian NSF contract KP-06-Russia/33.

\end{document}